\documentclass[11pt]{amsart}

\usepackage{amsmath,amssymb,pb-diagram,pst-all,graphicx}
\usepackage{amscd}
\usepackage[english]{babel} 
\usepackage{hyperref}     
\usepackage{geometry}             
\DeclareMathOperator{\Ind}{Ind}
\DeclareMathOperator{\lks}{lks}
\DeclareMathOperator{\HH}{H}

\newcommand{\ZZ}{\mathbb Z}
\newcommand{\PP}{\mathbb P}
\newcommand{\CC}{\mathbb C}
\newcommand{\NN}{\mathbb N}
\newcommand{\mcB}{\mathcal B}
\newcommand{\mcC}{\mathcal C}
\newcommand{\mcI}{\mathcal I}
\newcommand{\mcD}{\mathcal D}
\newcommand{\mcL}{\mathcal L}
\newcommand{\redu}{{\mathrm {red}}}
\newcommand{\ucurve}{\underline{\mathrm {Curve}}}
\newcommand{\Cov}{\mathop {\rm Cov}\nolimits}
\newcommand{\Sub}{\underline {\mathrm {Sub}}}

\newtheorem{thm}{Theorem}[section] 
\newtheorem{cor}[thm]{Corollary}

\newtheorem{prop}[thm]{Proposition}
\newtheorem{defin}[thm]{Definition}
\newtheorem{rem}[thm]{Remark}

\begin{document}

\title[Topology of $k$-Artal arrangments]{On the topology of arrangements of a cubic and its inflectional tangents}

\author[S. Bannai]{Shinzo Bannai}
\address{Department of Natural Sciences, National Institute of Technology, Ibaraki College, 866 Nakane, Hatachnaka, Ibaraki 312-8508, Japan}

\author[B. Guerville-Ball\'e]{Beno\^it Guerville-Ball\'e}
\address{Department of Mathematics, Tokyo Gakugei University, Koganei-shi, 
Tokyo 184-8501, Japan}

\author[T. Shirane]{Taketo Shirane}
\address{National Institute of Technology, Ube College, 2-14-1 Tokiwadai, Ube 755-8555, Yamaguchi, Japan}

\author[H. Tokunaga]{Hiro-o Tokunaga}
\address{Departiment of Mathematicas and Information Sciences, Tokyo Metropolitan University, Minami-Ohsawa 1-1, Hachoji, Tokyo 192-0397, Japan}


\begin{abstract}
	A $k$-Artal arrangement is a reducible algebraic curve composed of a smooth cubic and $k$ inflectional tangents. By studying the topological properties of their subarrangements, we prove that for $k=3,4,5,6$, there exist Zariski pairs of $k$-Artal arrangements. These Zariki pairs can be distinguished in a geometric way by the number of collinear triples in the set of singular points contained in the cubic.
\end{abstract}

\maketitle

\section{Introduction}

In this article, we continue to study Zariski pairs for reducible plane curves
based on the idea used in \cite{bannai-tokunaga}. 
A pair $(\mcB^1, \mcB^2)$ of reduced plane curves in $\PP^2$ is said to be a 
Zariski pair if (i) both $\mcB^1$ and $\mcB^2$ have the same combinatorics and 
(ii) $(\PP^2, \mcB^1)$ is {\it not}  homeomorphic to $(\PP^2, \mcB^2)$ (see \cite{act} for
details about Zariski pairs). As we have seen in \cite{act}, the study of Zariski pairs, roughly 
speaking, consists of two steps:

\begin{enumerate}
\item[(i)] How to construct (or find) plane curves with the same combinatorics but having {\it some} different properties.

\item[(ii)]  How to distinguish the topology of $(\PP^2, \mcB^1)$ and $(\PP^2, \mcB^2)$.
\end{enumerate}

As for the second step, various tools such as fundamental groups, Alexander invaritants,
braid monodoromies, existence/non-existence of Galois covers and so on have been
used.
In \cite{bannai-tokunaga}, the first and last authors considered another elementary
method in order to study
  Zariski $k$-plets for  arrangements of reduced plane curves and showed its effeciveness
  by giving some new examples. In this article, we study the toplogy of 
  arrangements of a smooth cubic and its inflectional tangents along the same line.

\subsection{Subarrangements}  
 
 We here reformulate our idea in \cite{bannai-tokunaga} more precisely.
 Let $\mcB_o$ be a (possibly empty) reduced plane curve $\mcB_o$. We define $\ucurve_{\redu}^{\mcB_o}$ to be the set of the reduced plane curves of the form
 $\mcB_o + \mcB$, where $\mcB$ is a reduced curve with no common component with $\mcB_o$.
 
 Let  $\mcB = \mcB_1 + \cdots + \mcB_r$ denote the irreducible decomposition of $\mcB$.  For a subset $\mcI$ of the power set 
 $2^{\{1, \ldots, r\}}$ of 
$\{1, \ldots, r\}$,  which does not contain the empty set $\emptyset$, we define  the sub set $\Sub_{\mcI}(\mcB_o, \mcB)$ of $\ucurve_{\redu}^{\mcB_o}$ by:
\[
\Sub_{\mcI}(\mcB_o, \mcB):=\left . \left \{\mcB_o + \sum_{i \in I} \mcB_i \,\, \right | \,\, I \in \mcI \right \}.
\]
For $\mcI = 2^{\{1, \ldots, r\}}\setminus \emptyset$, we denote 
 $\Sub(\mcB_o, \mcB) = \Sub_{\mcI}(\mcB_o, \mcB)$. 

 Let $A$ be a set and suppose that a map 
\begin{equation*}
	\Phi_{\mcB_o} : \ucurve_{\redu}^{\mcB_o} \to A
\end{equation*}
with the following property is given: for $\mcB_o + \mcB^1, \mcB_o + \mcB^2 \in \ucurve_{\redu}^{\mcB_o}$, if
 there exists a homeomorphism $h : (\PP^2, \mcB_o + \mcB^1) \to (\PP^2, \mcB_o + \mcB^2)$ with $h(\mcB_o) = \mcB_o$, 
 then $\Phi_{\mcB_o}(\mcB_o + \mcB^1) = \Phi_{\mcB_o}(\mcB_o + \mcB^2)$. \\

We denote by $\tilde{\Phi}_{\mcB_o, \mcB}$ the restriction of $\Phi_{\mcB_o}$ to $\Sub(\mcB_o, \mcB)$.
Note that if 
 there exists a homeomorphism $h : (\PP^2, \mcB_o + \mcB^1) \to (\PP^2, \mcB_o + \mcB^2)$ for 
 $\mcB_o + \mcB^1, \mcB_o + \mcB^2 \in \ucurve_{\redu}^{\mcB_o}$ with $h(\mcB_o) = \mcB_o$,
  then we have the induced map $h_{\natural} :  \underline{{\rm Sub}}(\mcB_o, \mcB^1) \to 
  \underline{{\rm Sub}}(\mcB_o, \mcB^2)$ such that 
$\tilde{\Phi}_{\mcB_o, \mcB^1} = \tilde{\Phi}_{\mcB_o, \mcB^2}\circ h_{\natural}$: 
\[
\begin{diagram}
\node{\underline{{\rm Sub}}(\mcB_o, \mcB^1)}\arrow{se,t}{\tilde{\Phi}_{\mcB_o, \mcB^1}}\arrow{s,l}{h_{\natural}} \\
\node{\underline{{\rm Sub}}(\mcB_o, \mcB^2)}\arrow{e,b}{ \tilde{\Phi}_{\mcB_o, \mcB^2}}\node{A} 
\end{diagram}
\]

\begin{rem}
	{\rm In~\S~\ref{sec:Phi_examples} we give four explicit examples for $\Phi_{\mcB_o}$ and $\tilde{\Phi}_{\mcB_o, \mcB}$ allowing to distinguish the $k$-Artal arrangements (see~\S~\ref{sec:Artal_arr} for the definition), using the Alexander polynomial, the existence of $D_6$-covers, the splitting numbers and the linking set.}\\
\end{rem}

 If $\mcD_o + \mcD^1, \mcD_o + \mcD^2\in\ucurve_{\redu}^{\mcD_o}$ have same the same combinatorics, 
 then any homeomorphism $h:(\mathcal{T}^1, \mcD_o + \mcD^1) \to (\mathcal{T}^2, \mcD_o + \mcD^1)$ induces a map $h_{\natural} :  \underline{{\rm Sub}}(\mcD_o, \mcD^1) \to \underline{{\rm Sub}}(\mcD_o, \mcD^2)$, 
where $\mathcal{T}^i$ is a tubular neighborhood of $\mcD_o + \mcD^i$ for $i=1,2$. 
 Let $(\mcD_o + \mcD^1, \mcD_o + \mcD^2)$ be a Zariski pair of curves in $\ucurve_{\redu}^{\mcD_o}$ such that
 
 \begin{itemize}
  \item  it is distinguished by $\Phi_{\mcD_o}$, i.e., any homeomorphism $h: (\mathcal{T}^1, \mcD_o + \mcD^1) \to (\mathcal{T}^2, \mcD_o + \mcD^2)$ necessarily satisfies
  $h(\mcD_o) = \mcD_o$ and $\Phi_{\mcD_o}(\mcD_o + \mcD^1) \neq \Phi_{\mcD_o}(\mcD_o + \mcD^2)$, and
  \item the combinatorial type of $\mcD_o + \mcD^1$ and $\mcD_o + \mcD^2$ is $\underline{\mathrm C}$.
  \end{itemize}
  
  Assuming the existence of such a Zariski pair for the combinatorial type $\underline{\mathrm C}$, we construct Zariski pair with glued combinatorial type.
  We first note that the following proposition is immediate:
  
  \begin{prop}\label{prop:main} Choose $\mcB_o + \mcB^1, \mcB_o + \mcB^2 \in \ucurve_{\redu}^{\mcB_o}$ with same
  combinatorial type. Let $\Sub_{\underline{\mathrm C}}(\mcB_o, \mcB^j)$ ($j = 1, 2$) be the sets of subarrangements of
  $\mcB_o + \mcB^j$ having the combinatorial type $\underline{\mathrm C}$ ($j = 1, 2$), respectively. If
  \begin{enumerate}
  \item[(i)] any homeomorphism $ h : (\mathcal{T}^1, \mcB_o + \mcB^1) \to (\mathcal{T}^2, \mcB_o + \mcB^2)$ necessarily satisfies $h(\mcB_o) = \mcB_o$, where $\mathcal{T}^i$ is a tubular neighborhood of $\mcB_o+\mcB^i$ for $i=1,2$, and
  \item[(ii)]   for some element $a_1\in A$, 
	\begin{equation*}
		\sharp (\tilde{\Phi}_{\mcB_o, \mcB^1}^{-1}(a_1)\cap \Sub_{\underline {\mathrm C}}(\mcB_o, \mcB^1) ) \\
  \neq 
   \sharp (\tilde{\Phi}_{\mcB_o, \mcB^2}^{-1}(a_1)\cap \Sub_{\underline {\mathrm C}}(\mcB_o, \mcB^2) ), 
	\end{equation*}
\end{enumerate}
then $(\mcB_o + \mcB^1, \mcB_o + \mcB^2)$ is a Zariski pair.  
  \end{prop}
	
\begin{rem}
	{\rm If for all automorphism $\sigma$ of the combinatorics of $\mcB_o+\mcB^j$, $\sigma(\mcB_o)=\mcB_o$ then hypothesis (i) of Proposition~\ref{prop:main} is always verified. In particular, it is the case if $\deg(\mcB_o)\neq\deg(\mcB_i)$, for $i=1,\dots,r$.}
\end{rem}

\subsection{Artal arrangements}\label{sec:Artal_arr}

In this article,  we apply Proposition~\ref{prop:main} to distinguish Zariski pairs formed by {\it Artal arrangements}. These curves are defined as follows:

Let $E$ be a smooth cubic, let $P_i$ ($1 \le i \le 9$) be its $9$ inflection points and let $L_{P_i}$ be the tangent lines at $P_i$ ($1 \le i \le 9$), respectively.

\begin{defin}\label{def:artal-arrangement}{ \rm Choose a subset $I \subset \{1, \ldots, 9\}$.  We call an arrangement $\mcC = E + \sum_{i \in I}L_{P_i}$ an Artal
arrangement for $I$. In particular, if $k = \sharp (I)$, we call $\mcC$ a $k$-Artal arrangement. 
}
\end{defin}

The idea is to apply Proposition~\ref{prop:main} to the case when $\mcB_o = E$ and $\mcB = \sum_{i\in I}L_{P_i}$. Let
$\mcC^1$ and $\mcC^2$ be two $k$-Artal arrangements. Note that if there exists homeomorphism $h: (\PP^2, \mcC^1) 
\to (\PP^2, \mcC^2)$, $h(E) = E$ always holds. 
In \cite{artal}, E.~Artal Bartolo gave an example of a Zariski
pair for $3$-Artal arrangements. Based on this example, we make use of our method to find other examples of Zariski pairs of $k$-Artal arrangement and obtain the following:

\begin{thm}\label{thm:artal-arrange}{There exists Zariski pairs for $k$-Artal arrangement for $k = 4, 5, 6$.
}
\end{thm}

\begin{rem}
{\rm Note that the case of $k = 5$ is considered in \cite{benoit-jb}. In \cite{benoit-jb}, it is shown that there exists an Zariski pair for $5$-Artal arrangement. 
}
\end{rem}

\section{Some explicit examples for $\Phi_{\mcB_o}$}\label{sec:Phi_examples}

We here introduce four examples for $\Phi_{\mcB_o}$. The last two were recently
considered by the second author, Meilhan \cite{benoit-jb} and the third author \cite{shirane},
respectively.
 
\subsection{$D_{2p}$-covers}

 For terminologies and notation, we use those introduced in  \cite{act}, \S 3 freely.
 
 Let $D_{2p}$  be the dihedral group of order $2p$. Let $\Cov_b(\PP^2, 2\mcB, D_{2p})$ be the set of isomorphism classes of $D_{2p}$-covers branched at $2\mcB$. 
 
 We now define $\Phi^{D_{2p}}_{\mcB_o} : \ucurve_{\redu}^{\mcB_o} \to \{0,1\}$
as follows:
\[
\Phi^{D_{2p}}_{\mcB_o}(\mcB_o + \mcB) = \left \{\begin{array}{cc}
                              1 & \mbox{if $\Cov_b(\PP^2, 2(\mcB_o + \mcB), D_{2p}) \neq \emptyset$} \\
                              0 &  \mbox{if $\Cov_b(\PP^2, 2(\mcB_o + \mcB), D_{2p})  = \emptyset$} 
                              \end{array}
\right .
\]

Note that $\Phi^{D_{2p}}_{\mcB_o}$ satisfies the required condition described in the Introduction. Thus, we define the map $\tilde{\Phi}^{D_6}_{\mcB_o,\mcB}:$ as the restriction of $\Phi^{D_6}_{\mcB_o}$ to $\Sub(\mcB_o, \mcB)$.

\subsection{Alexander polynomials}

For the Alexander polynomials of reduced plane curves, see~\cite{act}, \S~2. Let $\Delta:\ucurve_{\redu}^{\mcB_o}\rightarrow \CC[t]$ be the map assigning to a curve of $\ucurve_{\redu}^{\mcB_o}$ its Alexander polynomial. We define the map $\Phi^{\mathrm{Alex}}_{\mcB_o}:\ucurve_{\redu}^{\mcB_o}\rightarrow\left\{0,1\right\}$ by:
\begin{equation*}
	\Phi^{\mathrm{Alex}}_{\mcB_o}(\mcB_o + \mcB' ) = 
	\left \{
		\begin{array}{cl}
			1 & \mbox{if $\Delta(\mcB_o + \mcB') \neq 1$} \\
			0 & \mbox{if $\Delta(\mcB_o + \mcB') = 1.$}
		\end{array}
	\right .
\end{equation*}
As previously, we define $\tilde{\Phi}^{\mathrm{Alex}}_{\mcB_o,\mcB}$ as the restriction of $\Phi^{\mathrm{Alex}}$ to $\Sub(\mcB_o, \mcB)$.

\subsection{Splitting numbers}

Let $\mcB_o+\mcB$ be a reduced curves such that $\mcB_o$ is smooth. Let $\pi_m:X\rightarrow\PP^2$ be the unique cover branched over $\mcB$, corresponding to the surjection of $\pi_1(\PP^2\setminus\mcB)\rightarrow\ZZ/m\ZZ$ sending all meridians of the $\mcB_i$ to $1$. The \emph{splitting number} of $\mcB_o$ for $\pi_m$, denoted by $s_{\pi_m}(\mcB)$ is the number of irreducible component of the pull-back $\pi_m^*\mcB_o$ of $\mcB_o$ by $\pi_m$ (see~\cite{shirane} for the general definition). By~\cite[Proposition~1.3]{shirane}, the application:
\begin{equation*}
	\Phi^{\mathrm{split}}_{\mcB_o} :
	\left\{
		\begin{array}{ccc}
			\ucurve_{\redu}^{\mcB_o} & \longrightarrow & \NN^* \\
			\mcB_o+\mcB & \longmapsto & s_{\pi_3}(\mcB_o)
		\end{array}
	\right. ,
\end{equation*}
verify the condition of Proposition~\ref{prop:main}. We can then define the map $\tilde{\Phi}^{\mathrm{split}}_{\mcB_o,\mcB}:\Sub(\mcB_o, \mcB)\rightarrow\NN^*$ as the restriction of $\Phi^{\mathrm{split}}_{\mcB_o}$ to $\Sub(\mcB_o, \mcB)$.

\subsection{Linking set}

Let $\mcB_o$ be a non-empty curve, with smooth irreducible components. A cycle of $\mcB_o$ is an $S^1$ embedded in $\mcB_o$. For $\mcB_o+\mcB\in\ucurve_{\redu}^{\mcB_o}$, we define the \emph{linking set} of $\mcB_o$, denoted by $\lks_{\mcB}(\mcB_o)$, as the set of classes in $\HH_1(\PP^2\setminus\mcB)/\Ind_{\mcB_o}$ of the cycles of $\mcB_o$ which not intersect $\mcB$, where $\Ind_{\mcB_o}$ is the subgroup of $\HH_1(\PP^2\setminus\mcB)$ generated by the meridians in $\mcB_o$ around the points of $\mcB_o \cap \mcB$. This definition is weaker than~\cite[Definition~3.9]{benoit-jb}. By \cite[Theorem~3.13]{benoit-jb}, the map defined by:
\begin{equation*}
	\Phi^{\mathrm{lks}}_{\mcB_o}:
	\left\{
		\begin{array}{ccc}
			\ucurve_{\redu}^{\mcB_o} & \longrightarrow & \NN^* \cup\left\{ \infty \right\} \\
			\mcB_o+\mcB & \longmapsto & \sharp \lks_\mcB(\mcB_o)
		\end{array}
	\right.
\end{equation*}
verify the condition of Proposition~\ref{prop:main}. We can thus define the map 
$\tilde{\Phi}^{\mathrm{lks}}_{\mcB_o,\mcB}$ as the restriction of $\Phi^{\mathrm{lks}}_{\mcB_o}$ to $\Sub(\mcB_o, \mcB)$.

\section{The geometry of inflection points of a smooth cubic}\label{sec:geometry}

Let $E\subset \PP^2$ be a smooth cubic curve and let $O\in E$ be an inflection point of $E$. In this section we consider the elliptic curve $(E, O)$. The following facts are well known :

\begin{enumerate}
\item The set of inflection points of $E$ can be identified with $(\ZZ/3\ZZ)^{\oplus 2}\subset E$, the subgroup of three torsion points of $E$. 

\item Let $P, Q, R$ be distinct inflection points of $E$. Then $P, Q, R$ are collinear if and only if $P+Q+R=O\in (\ZZ/3\ZZ)^{\oplus 2}$.
\end{enumerate}

From the above facts we can study the geometry of inflection points and the following proposition follows:

\begin{prop} 
Let $E$ be a cubic curve and $\{P_1, \ldots, P_k\} \in E$ be a set of distinct inflection points of $E$. Let $n$ be the number of triples $\{P_{i_1}, P_{i_2}, P_{i_3}\}\subset \{P_1, \ldots, P_k\}$ such that they are collinear. Then the possible values of $n$ for $k=3, \ldots, 9$ are as in the following table: 
\begin{center}
\begin{tabular}{c|c|c|c|c|c|c|c}
$k$ & $3$ & $4$ & $5$ & $6$ & $7$ & $8$ & $9$ \\
\hline
$n$ & $0, 1$ & $0, 1$ & $1, 2$ & $2, 3$ & $5$ & $8$ & $12$
\end{tabular}
\end{center}
\end{prop}

\section{Proof of the Main Theorem}\label{sec:proof}

\subsection{The case of $3$-Artal arrangements}

Using the four invariants introduced in \S~\ref{sec:Phi_examples}, we can prove the original result of E.~Artal. Let $I=\{i_1,i_2,i_3\}\subset \{1,\ldots, 9\}$ and $\mcL_I=\sum_{i\in I}L_{P_i}$.
\begin{thm}\label{thm:artal-tokunaga}
For a $3$-Artal arrangement $\mcC = E + \sum_{i\in I}{ L_{P_i}}=E+\mcL_I$,  we have:
\begin{enumerate}
	\itemsep 0.25cm
	\item 
	$
	\tilde{\Phi}_{E, \mcL_I}^{D_6}(\mcC ) = 
	\left \{ 
		\begin{array}{cl}
			 1 & \mbox{if $P_{i_1}, P_{i_2}, P_{i_3}$ are collinear} \\
			 0 & \mbox{if $P_{i_1}, P_{i_2}, P_{i_3}$ are not collinear} 
		\end{array} 
	\right.
	$
	\item 
	$
	\tilde{\Phi}_{E, \mcL_I}^{\mathrm{Alex}}(\mcC) = 
	\left \{ 
		\begin{array}{cl}
			 1 & \mbox{if $P_{i_1}, P_{i_2}, P_{i_3}$ are collinear} \\
			 0 & \mbox{if $P_{i_1}, P_{i_2}, P_{i_3}$ are not collinear} 
		\end{array} 
	\right.
	$
	\item 
	$
	\tilde{\Phi}_{E, \mcL_I}^{\mathrm{split}}(\mcC) = 
	\left \{ 
		\begin{array}{cl}
			 3 & \mbox{if $P_{i_1}, P_{i_2}, P_{i_3}$ are collinear} \\
			 1 & \mbox{if $P_{i_1}, P_{i_2}, P_{i_3}$ are not collinear} 
		\end{array} 
	\right.
	$
	\item 
	$
	\tilde{\Phi}_{E, \mcL_I}^{\mathrm{lks}}(\mcC) = 
	\left \{ 
		\begin{array}{cl}
			 1 & \mbox{if $P_{i_1}, P_{i_2}, P_{i_3}$ are collinear} \\
			 3 & \mbox{if $P_{i_1}, P_{i_2}, P_{i_3}$ are not collinear} 
		\end{array} 
	\right.
	$
\end{enumerate}
\end{thm}

\begin{proof}
	\begin{enumerate}
	  \item This is the result of the last author (\cite{tokunaga96}).
		\item This is the result of E.~Artal (\cite{artal}).
		\item By \cite[Theorem~2.7]{shirane}, we obtain $\Phi^{\mathrm{split}}(\mcC,E)=3$ if the three tangent points are collinear, and $\Phi^{\mathrm{split}}(\mcC,E)=1$ otherwise. 
		\item Using the same arguments as in~\cite{guerville-shirane}, we can prove that, in the case of $3$-Artal arrangement, $s_{\mcB_o}(C)=\frac{3}{\sharp \lks_\mcB(\mcB_o)}$. Using the previous point we obtain the result.
	\end{enumerate}
\end{proof}

\begin{rem}
{\rm It is also possible to consider $\tilde{\Phi}_{\mcL_J, E}^{\mathrm{lks}}(E)$. But in this case, we have no method to compute it in the general case. But, if $C$ is the cubic defined by $x^3 - xz^2 - y^2z = 0$, the computation done in~\cite{benoit-jb} implies the result.}
\end{rem}

\begin{cor}\label{cor:artal-tokunaga}{ Choose $\{i_1, i_2, i_3, i_4\} \subset \{1, \ldots, 9\}$ such that
$P_{i_1}, P_{i_2}, P_{i_3}$ are collinear, while $P_{i_1}, P_{i_2}$ and $P_{i_4}$ are not collinear. Put
$\mcC_1 = E + L_{i_1} + L_{i_2}  + L_{i_3}$ and $\mcC_2 = E + L_{i_1} + L_{i_2}  + L_{i_4}$. Then
$(\mcC_1, \mcC_2)$ is a Zariski pair.
}
\end{cor}

\subsection{The other cases}

Choose a subset $J$ of $\subset \{1, \ldots, 9\}$ such that $4\leq \sharp J \leq 6$ and let 
\[
\mcC := E + \mcL_J,  \, \mcL_J = \sum_{j \in J} L_{P_j},
\]
be a $k$-Artal arrangement. To distinguish these arrangements in a geometric way (as the collinearity in the case of $3$-Artal arrangement), let us introduce the type of a $k$-Artal arrangement.
 
\begin{defin}
For $k= 4 ,5, 6$,  we say an arrangement of the form  $\mathcal{C}=E+L_{P_1}+\cdots+L_{P_k}$ to be of Type I if the number $n$ of collinear triples in $\{P_1, \ldots, P_k\}$ is $n= k - 3$, while  we say  $\mathcal{C}$ to be  of Type II
if the number $n$ of collinear triples in $\{P_1, \ldots, P_k\}$ is $n= k - 4$.
\end{defin}

\begin{thm}
Let $\mathcal{C}_1$ be an arrangement of Type I and $\mathcal{C}_2$ be an arrangement of Type II. Then $(\PP^2, \mathcal{C}_1)$ and $(\PP^2, \mathcal{C}_2)$ are not homeomorphic as pairs.

Furthermore if $\mathcal{C}_1$ and $\mathcal{C}_2$ have the same combinatorics, $\mathcal{C}_1, \mathcal{C}_2$ form a Zariski pair.
\end{thm}

\begin{proof}
	Let $\mcC$ be a $k$-Artal arrangement ($k = 4, 5, 6$). We denote by $\Sub(E, \mcL_J)_3$ the set of $3$-Artal arrangements contained in $\Sub(E, \mcL_J)$. Let $\Phi_{\mcC, 3}^{D_6}$,  $\Phi_{\mcC, 3}^{\mathrm{Alex}}$ $\Phi_{\mcC, 3}^{\mathrm{split}}$ and $\Phi_{\mcC, 3}^{\lks}$ be the restrictions  of $\tilde\Phi_{E,\mcL_J}^{D_6}$, $\tilde\Phi_{E,\mcL_J}^{\mathrm{Alex}}$, $\tilde\Phi_{E,\mcL_J}^{\mathrm{split}}$ and $\tilde\Phi_{E,\mcL_J}^{\lks}$ to $\Sub(E,\mcL_J)_3$, respectively.
	Then by Theorem~\ref{thm:artal-tokunaga}, we have
	\begin{equation*}
		\sharp ({\Phi_{\mcC, 3}^{D_6}}^{-1}(1) ) \\
		= \sharp ({\Phi_{\mcC, 3}^{\mathrm{Alex}}}^{-1}(1) ) \\ 
		=\sharp ({\Phi_{\mcC, 3}^{\mathrm{split}}}^{-1}(3)) \\
		= \sharp ({\Phi_{\mcC, 3}^{\lks}}^{-1}(1))  \\
		= \left \{
		 \begin{array}{cl}
				k-3 & \mbox{if $\mcC$ is type I} \\
				k-4 & \mbox{if $\mcC$ is type II.}
			\end{array}
		\right .
	\end{equation*}
	If  a homeomorhism $h: (\PP^2, \mathcal{C}_1) \to (\PP^2, \mathcal{C}_2)$ exists, it
	follows that $h_{\natural}(\Sub(\mcC_1)_3) =\Sub(\mcC_2)_3$.
	This contradicts the above values. Hence our statements follow.
\end{proof}

\begin{rem}
{\rm As a final remark, we note that for $k=1, 2, 7, 8, 9$ it can be proved that there do not  exist Zariski pairs consisting of $k$-Artal arrangenents.}
\end{rem}


\end{document}